\documentclass[12pt,a4paper]{amsart}
\usepackage{amsfonts}%
\usepackage{amsmath}
\usepackage{graphicx} 

\usepackage{amssymb}
\numberwithin{equation}{section}
\newtheorem{theorem}{Theorem.}[section]

\newtheorem{lemma}[theorem]{Lemma}

\begin{document}

\title[]{Permutation-Invariance in Koml\'{o}s' Theorem for Non-negative Random
Variables}

\author[Abdessamad DEHAJ]{ A. DEHAJ \textsuperscript{1}}

\address{\textsuperscript{1}Faculty of Sciences Ben M'Sik, Hassan II University of Casablanca. Casablanca, Morocco.}
\email{a.dehaj@gmail.com}

\author[Mohamed Guessous]{ M. Guessous \textsuperscript{1}}
\address{\textsuperscript{1}Faculty of Sciences Ben M'Sik, Hassan II University of Casablanca. Casablanca, Morocco.}
\email{guessousjssous@yahoo.fr}

\author[Noureddine Sabiri]{ N. Sabiri \textsuperscript{1}}
\address{\textsuperscript{1}Faculty of Sciences Ben M'Sik, Hassan II University of Casablanca. Casablanca, Morocco.}
\email{sabiri.noureddine@gmail.com}



\begin{abstract}
We provide a permutation-invariant version of the Koml\'{o}s' theorem for non-negative random variables. The proof is quite elementary in the sense that it did not use the Axiom of Choice, and was based on a recent result in \cite{D:2021}.
\end{abstract}

\maketitle   

\section{Introduction}
In a given probability space $(\Omega, \Sigma ,\mu )$ the renowned Koml\'{o}s' theorem \cite{K:1967} affirms that every bounded sequence of $L^{1}_{\mathbb{R}}(\mu)$ has a subsequence which along with all its subsequence converges $a.e.$ in Césaro mean to an integrable function $f$ defined on the same probability space. This theorem of Koml\'{o}s is well-known and has received several generalizations on diverse directions. For example Berkes \cite{B:1990} proved that every bounded sequence of $L^{1}_{\mathbb{R}}(\mu)$ has a subsequence so that all its permutation are Cesàro converging a.s. to an integrable function. Recently, the first and the second authors of the present paper have proved that the almost surely Cesàro convergence in Komlós theorem holds true for every permutation of any subsequence \cite{D:2021}. As indicated in \cite{W:2004} F. Delbaen and Y. Kabanov asked whether the $L^{1}_{\mathbb{R}}(\mu)$-boundedness condition can be removed for non-negative sequence of random variables, if the limit function in Koml\'{o}s statement allows to take the value $\infty$. Von Weizsacker \cite{W:2004} answered affirmatively by proving a plausible extension of the Koml\'{o}s-Berkes theorem to sequences of non-negative random variables: every sequence of non-negative measurable functions $(f_{n})$ on $(\Omega, \Sigma ,\mu)$ contains a subsequence $(f'_{n})$ which converges $a.e.$ in Césaro mean to some measurable function $f : \Omega \rightarrow \left[ 0, \infty \right]$ and this holds for every permutation of $(f'_{n})$. This becomes known as the Von Weizsacker theorem. Further, by modifying the proof of \cite{W:2004} and using the Koml\'{o}s theorem in \cite{K:1967} instead of the version in \cite{B:1990}, the statement of the Von Weizsacker theorem remains valid for every subsequence of $(f'_{n})$ in place of every permutation (see \cite{K:2009}, Thm. 5.2.3). In \cite{T:2021}, the author also gives a description of the set where the limit function $f$ is finite. There has been much recent work related to the Komlos' theorem, see for example \cite{S:2020, D:2020, G:1997}.
\\
The purpose of this paper is to show that the version of Koml\'{o}s theorem for non-negative measurable functions is invariant under permutation. More precisely, we will show that every sequence of non-negative measurable functions $(f_{n})$ has a subsequence $(f'_{n})$ such that $(h_{\pi(n)})$ converge $a.e.$ in Césaro mean to a measurable function $f$ with values in a $\mathbb{R}^{+}\cup { +\infty } $ for every subsequence $(h_{n})$ of $(f'_{n})$ and every permutation $(\pi (1),\pi (2),...)$ of $(1,2,...)$. More precisely, we have the following result.
\begin{theorem}
Let $(f_{n})_{n}$ be a sequence of non-negative random variables on $(\Omega,\Sigma ,\mu )$. Then, there exists a random variable $f$ with values in a $%
\mathbb{R}
^{+}\cup \left\{ +\infty \right\} $ and a subsequence $(g_{n})_{n}$ of $(f_{n})_{n}$ such that, for any subsequence $(h_{n})_{n}$ of $(g_{n})_{n}$
and for every permutation $(\pi (1),\pi (2),...)$ of $(1,2,...)$
\begin{equation*}
\lim_{n}S_{n}\left( h_{\pi \left( k\right) }\right) =f\text{ }\mu \text{-} a.e.
\end{equation*}
\end{theorem}
We develop similar arguments which some of them were used in Von Weizsacker paper but, we do not appeal to the Axiom of Choice as it was done in \cite{W:2004}. The method provides also an elementary proof for the original Von Weizsacker result.
\section{Notations and Preliminaries}
We denote by $\mathtt{Si}(\mathbb{N}^{\ast })$ the set of the strictly increasing sequences of natural numbers. For any $a\geq 0,$ $m\in \mathbb{N}^{\ast }$ and $\left( f_{n}\right) _{n}$ a sequence of random variables, we note
\begin{eqnarray*}
F_{a}(f_{n})\left( \omega \right) &=&\left\{ 
\begin{array}{ccc}
f_{n}(\omega) & \text{if} & \left\Vert f_{n}(\omega)\right\Vert <a \\ 
0 & \text{if not.} & 
\end{array}%
\right.  \\
S_{m}(f_{k}) &=&\frac{1}{m}\sum_{k=1}^{m}f_{k}
\end{eqnarray*}
We give in the following subsection three results which will be useful to prove our main result. The first one is a recent version of the Komlós' theorem from \cite{D:2021}. The second one is an useful lemma for our analysis, its proof is based on the similar arguments to the one given in \cite{W:2004}. The third one is a simple extraction Lemma.
\begin{theorem}\label{thm1}
Let $H$ be a Hilbert space and $(f_{n})_{n}$ a bounded sequence in $L_{H}^{1}\left( \Omega ,\Sigma,\mu \right) $. Then, there exists a subsequence $(g_{n})_{n}$ of $(f_{n})_{n}$ and $f$ in $L_{H}^{1}\left( \Omega ,\Sigma,\mu \right) $ such that
\begin{equation*}
\lim_{n}S_{n}\left( h_{\pi \left( k\right) }\right) \text{ converges to }f%
\text{ }%
{\mu}%
\text{-}a.e.,
\end{equation*}
for any subsequence $(h_{n})_{n}$ of $(g_{n})_{n}$ and for every permutation $(\pi (1),\pi (2),...)$ of $(1,2,...)$.
\end{theorem}
The following lemma generalizes Lemma 1 and Lemma 2 in \cite{W:2004}, because the result include all permutation of any subsequence of $(g_{n})_{n}$.
\begin{lemma}\label{lem2}
Let $(f_{n})_{n}$ be a sequence of non-negative random variables on $(\Omega,\Sigma ,\mu )$. Assume that
\begin{equation*}
b:=\lim_{k}\sup_{n}\mu ( \left\lbrace f_{n}\geq k \right\rbrace )>0.
\end{equation*}
Then, there exists a non-negligible measurable set $C$ and a subsequence $(g_{n})_{n}$ of $(f_{n})_{n}$ such that for every $ \psi \in \mathtt{Si}(\mathbb{N}^{\ast })$ and a permutation $(\pi (1),\pi(2),...)$ of $\ (1,2,...)$
\begin{eqnarray}
\lim_{n}S_{n}\left( g_{\psi \circ \pi \left( k\right) }\right) &=&+\infty \text{ on }C
\\
\lim_{k}\sup_{n}\mu \left( g_{n}1_{C^{c}}\geq k\right) &=&0.
\end{eqnarray}
\end{lemma}

\begin{proof}
We have $\sup_{n}\mu \left( f_{n}\geq k\right) \geq b >0$ for all $k$, then there exists a  subsequence $\left( f_{n}^{1}\right) _{n}$ of $(f_{n})_{n}$ such that for every $n\in \mathbb{N}$
\begin{equation*}
\mu \left( f_{n}^{1}\geq n\right) \geq \frac{b}{2}.
\end{equation*}
Let $C_{n}=\left\{ f_{n}^{1}\geq n\right\} (n\geq1)$. Applying Theorem~\ref{thm1} to the sequence $\left( 1_{C_{n}}\right) _{n}$ there exists $g\in L_{%
\mathbb{R}
}^{1}$ and $\varphi \in \mathtt{Si}\left( 
\mathbb{N}^{\ast }
\right)$ such that for any $\psi \in \mathtt{Si}\left( \mathbb{N} ^{\ast }\right)$ and every permutation $(\pi (1),\pi (2),...)$ of $(1,2,...)$ we have
\begin{equation*}
\lim_{n} S_{n}\left( 1_{C_{\varphi \circ \psi \circ \pi (k)}}\right) =g\text{ }%
\mu \text{-a.e.,}
\end{equation*}
and by the D.C.T we get 
\begin{equation}\label{eq1}
\forall A \in \Sigma , \quad \lim_{n} \int_{A} S_{n}\left( 1_{C_{\varphi \circ \psi \circ \pi (k)}}\right) d\mu =\int_{A} g d\mu.
\end{equation}
Let $C=\left\{ g>0\right\}$ and $g_{n}=f^{1}_{ \varphi (n)}$. Since $g\leq 1$ $\mu $-a.e., we get by (\ref{eq1})
\begin{equation*}
\mu ( C ) \geq \int g d\mu =\lim_{n}\int
S_{n}\left( 1_{C_{\varphi (k)}}\right) d\mu \geq \frac{b}{2},
\end{equation*}
therefore
\begin{equation*}
\mu \left( C\right) > 0.
\end{equation*}
Let $\psi \in \mathtt{Si}\left(\mathbb{N}^{*}\right)$, a permutation $(\pi (1),\pi (2),...)$ of $(1,2,...)$ and every $r>0$. One has for almost everywhere $\omega \in \Omega$
\begin{eqnarray*}
\underset{n}{\underline{\lim }} \: S_{n}(g_{ \psi \circ \pi  (k)})(w) &\geq &%
\underset{n}{\underline{\lim }} \: S_{n}(g_{  \psi \circ \pi  (k)}1_{C_{\varphi \circ \psi \circ \pi (k)}})(w) \\
&\geq &r \: \underset{n}{\underline{\lim }} \: S_{n}(1_{C_{ \varphi \circ \psi \circ \pi (k)}}) (w)\\
&=&rg(w),
\end{eqnarray*}
As $g >  0$ on $C$ we get
\begin{equation*}
\underset{n}{\underline{\lim }} \: S_{n}(g_{ \psi \circ \pi (k)})=+\infty \text{ }%
\mu \text{-a.e. on }C.
\end{equation*}
To prove the second assertion let us note $d^{^{\prime }}=\lim_{k}\sup_{n}\mu \left\{ g_{\psi (n)}1_{C^{c}}\geq k\right\}$. If $d'>0$, by the same way, there exists a subsequence $(g^{1}_{n})$ of $(g_{\psi(n)})$ such that for all $n\in \mathbb{N}$
\begin{equation*}
\mu ( \left\lbrace  g^{1}_{n}1_{C^{c}}\geq n  \right\rbrace ) \geq \frac{d'}{2},
\end{equation*}
By (\ref{eq1}) we have 
\begin{equation*}
0=\lim_{n}\int_{C^{c}} S_{n}\left( 1_{C_{\varphi \circ \psi \circ \phi( k) }}\right) d\mu = \lim_{n} S_{n}(\mu ( \left\lbrace g^{1}_{k} 1_{C^{c}}\geq k \right\rbrace ) \geq \frac{d'}{2}
\end{equation*}
which is contradiction. Hence $d'=0$.
\end{proof}
\begin{lemma}\label{lem1}
Let $(f_{n})_{n}$ be a sequence of non-negative random variables on $(\Omega,\Sigma ,\mu )$ such that 
\begin{equation*}
\lim_{k}\sup_{n}\mu \left( f_{n}\geq k\right) =0.
\end{equation*}%
Then, there is a increasing sequence of positive real numbers $\left(p_{k}\right) _{k}$ such that 
\begin{equation*}
\underset{k\geq 1}{\sum }\mu \left( f_{k}\geq p_{k}\right) <+\infty.
\end{equation*}
Therefore,
\begin{equation*}
\lim_{k} \: ( f_{k}- F_{p_{k}}(f_{k})) =0 \qquad \mu-a.e.
\end{equation*}
\end{lemma}
\begin{proof}
For each $k \geq 1$ we construct a a increasing sequence $(p_{k})$  such that 
\begin{equation*}
\sup_{n}\mu \left( f_{n}\geq p_{k}\right) < \frac{1}{2^{k}}.
\end{equation*}
Hence,
\begin{equation*}
\underset{k\geq 1}{\sum }\mu \left( f_{k}\geq p_{k}\right) <+\infty.
\end{equation*}
\end{proof}
We close this section by the following useful lemma.
\begin{lemma}\label{lem}
Let $\nu$ any probability measure, $(f_{n})$ a sequence of $L_{\mathbb{R}}^{1}(\nu)$ with value in $\left[0, M \right]$ ($M>0$) and $f\in L_{\mathbb{R}}^{1}(\nu)$ such that 
\begin{equation*}
S_{n}(h_{k}) \underset{n} \rightarrow f \qquad a.e.
\end{equation*}
for every subsequence $(h_{n})$ of $(f_{n})$. Then $(f_{n})$ converges weakly to $f$ in $L_{\mathbb{R}}^{1}(\nu)$.
\end{lemma}
\begin{proof}
The D.C.T gives that $(S_{n}(h_{k}))_{n}$ converges to $f$ weakly in $L_{\mathbb{R}}^{1}(\nu)$ for every subsequence $(h_{n})$ of $(f_{n})$. Hence by an elementary property of Césaro convergence of real valued sequence we deduce that $(f_{n})$ converges to $f$ weakly in $L_{\mathbb{R}}^{1}(\nu)$.
\end{proof}
\section{ Proof of the Main Result}
For each $k\in \mathbb{N}^{\ast }$ the sequence $(F_{k}(f_{n}))_{n}$ is bounded in $L_{\mathbb{R}}^{1}$, so by Theorem~\ref{thm1} there is a sequence of random variable $(u_{k})_{k}$ and subsequences $\left(f_{n}^{1}\right) _{n},$ $\left( f_{n}^{2}\right) _{n},$ $\left(f_{n}^{3}\right) _{n},...,\left( f_{n}^{k}\right) _{n},...$ of $(f_{n})_{n}$, where $(f_{n}^{k+1})_{n}$ is a subsequence of $(f_{n}^{k})_{n}$ such that $(F_{k}(f_{n}))_{k}$ converges weakly to $u_{k}$ in $L_{\mathbb{R}}^{1}$ for all $k \geq 1$ and, for any permutation $(\pi (1),\pi (2),...)$ of $ (1,2,...)$
\begin{equation*}
\lim_{n}\frac{1}{n}\underset{p=1}{\overset{n}{\sum }}F_{k}(f_{\pi \left(p\right) }^{k}(\omega))= u_{k}(\omega)\text{ }\mu \text{-a.e., for all }k\in 
\mathbb{N}
^{\ast }.\text{ }
\end{equation*}
Moreover, this equality remains valid if we replace $(f_{n}^{k})_{n}$ by any of it's subsequences. Put $f'_{n}=f_{n}^{n} (n\geq 1)$, then for all $k \geq 1$ 
\begin{equation*}
\lim_{n}\frac{1}{n}\underset{p=1}{\overset{n}{\sum }}F_{k}(h_{\pi \left(p\right) }(\omega))=u_{k}(\omega)\text{ }\mu \text{-a.e.},
\end{equation*}
where $(h_{n})_{n}$ is a subsequence of $(f_{n}^{'})_{n}$ and $(\pi (1),\pi (2),...)$  any permutation of $ (1,2,...)$.
For all $k,n \geq 1$ we have $F_{k+1}(f_{n}) \geq F_{k}(f_{n})$ and since $(f_{n})$ is positive, the sequence $\left( u_{k}\right) _{k}$ is decreasing $\mu $-a.e. so,
\begin{equation*}
\underset{n}{\underline{\lim }}\frac{1}{n}\underset{p=1}{\overset{n}{\sum }}%
h_{\pi \left( p\right) }(\omega)\text{ }\geq \underset{n}{\lim }\frac{%
1}{n}\underset{p=1}{\overset{n}{\sum }}F_{k}(h_{\pi \left( p\right) }(\omega))=u_{k}(\omega),\text{ }\mu \text{-a.e., for all }k\in 
\mathbb{N}
^{\ast },
\end{equation*}
and therefore for almost everywhere $\omega \in \Omega$
\begin{equation*}
\underset{n}{\underline{\lim }}\frac{1}{n}\underset{p=1}{\overset{n}{\sum }}%
h_{\pi \left( p\right) }(\omega)\text{ }\geq u(\omega):= \lim_{k} \: u_{k}(\omega) \in \left[ 0, +\infty \right].
\end{equation*}
Let $A = \left\lbrace u < +\infty \right\rbrace $ and $b=\lim_{k}\sup_{n}\mu \left( f'_{n}1_{A}\geq k\right)$ we have for almost everywhere $ \omega \in A^{c}$
\begin{equation*}
\underset{n}{\lim } \:S_{n}\left( h_{\pi \left( k\right) }\right) (\omega)=+\infty.
\end{equation*}
\textbf{If} $b>0,$ there exists by Lemma~\ref{lem2} a non-negligible measurable set $C$ and a subsequence $(g'_{n})_{n}$ of $(f'_{n})_{n}$, such that for any subsequence $(h_{n})_{n}$ of $(g_{n}')_{n}$ and
for every permutation $(\pi (1),\pi (2),...)$ of $ (1,2,...)$
\begin{eqnarray*}
\lim_{n}S_{n}\left( h_{\pi \left( k\right) }1_{A}\right)  &=&+\infty \text{
on }C, \\
\lim_{k}\sup_{n}\mu \left( h_{n}1_{A\cap C^{c}}\geq k\right)  &=&0.
\end{eqnarray*}
Applying Lemma~\ref{lem1} to the sequence $\left( 1_{A\cap C^{c}}g_{n}^{^{\prime}}\right)_{n}$, there exists an increasing sequence of positive real numbers $\left( p_{k}\right) _{k}$ such that 
\begin{equation*}
1_{A \cap C^{c}}(g'_{n}-F_{p_{k}}(g'_{n})) \underset{n}\rightarrow 0 \quad \mu-a.e.
\end{equation*}
Let $\nu $ be any probability measure equivalent to $\mu $ such that $1_{A}u\in L_{\mathbb{R}}^{1}\left( \nu \right)$.  Lemma~\ref{lem} gives that for all $k \geq 1$
\begin{equation*}
\underset{n}{\lim }\int_{A \cap C^{c}}F_{p_{k}}\left( g'_{n}\right) \: d\nu =\int_{A \cap C^{c}}u_{k} \: d\nu \leq \int_{A\cap C^{c}} u \: d\nu< M+1,
\end{equation*}
where $M = \int_{A} u \: d\nu$. Let $\varphi \in \mathtt{Si}(\mathbb{N}^{\ast })$ such that 
\begin{equation*}
\int_{A\cap C^{c}}F_{p_{k}}\left( g_{\varphi \left( k\right) }^{^{\prime
}}\right) d\nu \leq M+1\text{ for all }k\geq 1.
\end{equation*}
Since the sequence $\left( F_{p_{k}}\left( g'_{\varphi \left(k\right)} 1_{A\cap C^{c}}  \right) \right) _{k}$ is bounded in $L_{\mathbb{R}}^{1}\left( \nu \right)$, by Theorem~\ref{thm1} there exists $\psi \in \mathtt{Si}(\mathbb{N}^{\ast }) $ and $f\in L_{\mathbb{R}}^{1}\left( \nu \right)$ such that for all $\phi \in \mathtt{Si}(\mathbb{N}^{\ast })$ and for every permutation $(\pi (1),\pi (2),...)$ of $ (1,2,...)$
\begin{equation*}
\lim_{n}S_{n} ( F_{p_{\psi \circ \phi \circ \pi (k) }} ( g'_{\varphi \circ \psi \circ \phi \circ \pi ( k )} 1_{A\cap C^{c}}) ) =f\text{ \ }\mu \text{-}a.e.
\end{equation*}
As
\begin{equation*}
\underset{k\geq 1}{\sum }\mu \left( g'_{\varphi \circ \psi \circ \phi \circ \pi \left( k\right) } 1_{A\cap C^{c}} \geq p_{\psi \circ \phi \circ \pi \left( k\right) }\right) <+\infty ,
\end{equation*}
by the Borel--Cantelli lemma we deduce that
\begin{equation*}
\lim_{n}S_{n}\left( g'_{\varphi \circ \psi \circ \phi \circ \pi \left(k\right) }\right) =f\text{ }\mu \text{-}a.e.\text{ on }A\cap C^{c}.
\end{equation*}%
\textbf{If} $b=0$, we take $C$ the empty set.\\
\textbf{Conclusion:} Put $g_{n}=g'_{\varphi \circ \psi (n)}$. Then for all $\phi \in \mathtt{Si}(\mathbb{N}^{\ast })$ and for every permutation $(\pi (1),\pi (2),...)$ of $(1,2,...)$ 
\begin{eqnarray*}
\lim_{n}S_{n}\left( \text{ }g_{\phi \circ \pi \left( k\right) }\right)  &=&f%
\text{ }\mu \text{-}a.e.\text{ on }A\cap C^{c} \\
\lim_{n}S_{n}\left( \text{ }g_{\phi \circ \pi \left( k\right) }\right) 
&=&+\infty \text{ }\mu \text{-}a.e.\text{ on }A^{c} \\
\lim_{n}S_{n}\left( \text{ }g_{\phi \circ \pi \left( k\right) }\right) 
&=&+\infty \text{ }\mu \text{-}a.e.\text{ on }A\cap C.
\end{eqnarray*}
Thus complete the proof.

\end{document}